\documentclass{amsart}

\usepackage{amssymb}
\usepackage{dsfont}

\newtheorem{theorem}{Theorem}[section]

\newtheorem{proposition}[theorem]{Proposition}
\newtheorem{corollary}[theorem]{Corollary}
\newtheorem{lemma}[theorem]{Lemma}

\theoremstyle{definition}
\newtheorem{definition}[theorem]{Definition}

\theoremstyle{remark}

\numberwithin{equation}{section}

\newcommand{\nsubset}{\not\subset}
\newcommand{\enproof}{\hspace*{\stretch{1}}\qedsymbol}  

\newcommand{\C}{\mathcal{C}}
\newcommand{\set}[1]{\left\{#1\right\}}  
\newcommand{\closure}[1]{\overline{#1}}  

\newcommand{\reals}{\mathds{R}}    
\newcommand{\naturals}{\mathds{N}}    

\newcommand{\D}{\mathcal{D}}
\newcommand{\F}{\mathcal{F}}

\newcommand{\G}{\mathcal{G}}

\newcommand{\Pp}{\mathcal{P}}
\newcommand{\Q}{\mathcal{Q}}

\newcommand{\cardinal}[1]{\left|#1\right|}

\newcommand{\continuum}{\mathfrak{c}}

\newcommand{\onepointcompactification}[1]{#1^\flat}

\newcommand{\zero}{\textbf{Z}}

\newcommand{\limitpoints}{\partial}

\newcommand{\level}[2]{\partial^{(#2)}#1}

\begin{document}

\title[Prime ideals in $\C_0(\reals)$]{Uncountable families of prime $z$-ideals in $\C_0(\reals)$}

\author{Hung Le Pham}
\address{Department of Mathematical and Statistical Sciences,
    University of Alberta, Edmonton, Alberta T6G 2G1, Canada}
\email{hlpham@math.ualberta.ca}



\thanks{This research is supported by a Killam Postdoctoral Fellowship and a Honorary PIMS Postdoctoral Fellowship}

\begin{abstract}
Denote by $\continuum=2^{\aleph_0}$ the cardinal of continuum. We
construct an intriguing family $(P_\alpha:\ \alpha\in\continuum)$
of prime $z$-ideals in $\C_0(\reals)$ with the following
properties:
\begin{itemize}
    \item If $f\in P_{i_0}$ for some $i_0\in\continuum$, then
        $f\in P_i$ for all but finitely many $i\in \continuum$;
        \item $\bigcap_{i\neq i_0} P_i \nsubset P_{i_0}$ for each $\i_0\in \continuum$.
\end{itemize}
We also construct a well-ordered increasing chain, as well as a
well-ordered decreasing chain, of order type $\kappa$ of prime
$z$-ideals in $\C_0(\reals)$ for any ordinal $\kappa$ of
cardinality $\continuum$.
\end{abstract}

\maketitle

\section{Introduction}
\label{introduction}

Let $\Omega$ be a locally compact space. In \cite{pham2005b}, we
introduced the notion of pseudo-finite family of prime ideals as
follows.

\begin{definition}
    \label{pseudo_finiteness_definition}
An indexed family $(P_i)_{i\in S}$ of prime ideals in
$\C_0(\Omega)$ is \emph{pseudo-finite} if $f\in P_i$ for all but
finitely many $i\in S$ whenever $f\in \bigcup_{i\in S} P_i$.
\end{definition}
A pseudo-finite family $(P_i:\ i\in S)$ of prime ideals in
$\C_0(\Omega)$ has many interesting properties, for example, when
$S$ is infinite, the union $\bigcup_{i\in S} P_i$ is again a prime
ideal and any infinite subfamily of $(P_i)$ gives rise to the same
union.

A pseudo-finite family of prime ideals $(P_i:\ i\in S)$ is said to
be \emph{non-redundant} if for every proper subset $T$ of $S$,
$\bigcap_{i\in T} P_i\neq \bigcap_{i\in S} P_i$. Non-redundancy is
equivalent to either of the following (\cite[Lemma
3.4]{pham2005b}):
\begin{itemize}
        \item[(a)] $P_\alpha \nsubset P_{\beta}\quad(\alpha\neq \beta\in S)$;
        \item[(b)] $\bigcap_{\beta\neq\alpha} P_\beta \nsubset P_{\alpha}$ for each $\alpha\in S$.
\end{itemize}
Note that (a) is apparently weaker, whereas (b) is apparently
stronger than the non-redundancy. Thus, in this case,
$\bigcap_{i\in S} P_i$ cannot be written as the intersection of
less than $\cardinal{S}$ prime ideals and
$\cardinal{S}\le\cardinal{\C_0(\Omega)}$. Furthermore, for every
pseudo-finite family of prime ideals, the subfamily consisting of
those ideals that are minimal in the family is non-redundant and
pseudo-finite and has the same intersection as the original
family.

The notion of pseudo-finiteness has a connection with automatic
continuity theory. It is proved in \cite{pham2005b} that, assuming
the Continuum Hypothesis, for each pseudo-finite family $(P_i:\
i\in S)$ of prime ideals in $\C_0(\Omega)$ such that
$\cardinal{\C_0(\Omega)/\bigcap_{i\in S} P_i}=\continuum$, there
exists a homomorphism from $\C_0(\Omega)$ into a Banach algebra
whose continuity ideal  is $\bigcap_{i\in S} P_i$. Recall that the
continuity ideal is the largest ideal of $\C_0(\Omega)$ on which
the homomorphism is continuous, and it is always an intersection
of prime ideals (see \cite{dales2000} for more details).

Suppose that $\Omega$ is metrizable and that
$\level{\onepointcompactification{\Omega}}{\infty}\neq \emptyset$;
see \S\ref{basicdefinitions} for the definition. Examples of such
spaces include many countable locally compact spaces and all
uncountable locally compact Polish spaces. For such $\Omega$, it
is known that there exists an infinite non-redundant pseudo-finite
sequence of prime ideals in $\C_0(\Omega)$ (\cite{pham2005b}).
Here, we are going to show that there exists even a non-redundant
pseudo-finite family $(P_i: i\in\continuum)$ of prime ideals in
$\C_0(\Omega)$ (Theorem \ref{constructpsfprimes}). As a
consequence, assuming the Continuum Hypothesis, there exists a
homomorphism from $\C_0(\Omega)$ into a Banach algebra whose
continuity ideal cannot be written as intersection of countably
many prime ideals.

Note that when $\Omega$ is metrizable and
$\level{\onepointcompactification{\Omega}}{\infty}= \emptyset$,
then every non-redundant pseudo-finite families of prime ideals in
$\C_0(\Omega)$ is finite and the continuity ideal of every
homomorphism from $\C_0(\Omega)$ into a Banach algebra is the
intersection of finitely many prime ideals. (\cite{pham2005b})

In \S\ref{decreasingsection}, we shall construct, for every
ordinal $\kappa$ of cardinality at most $\continuum$, a
well-ordered decreasing chain of order type $\kappa$ of prime
$z$-filters on any uncountable locally compact Polish space. In
particular, there exists a well-ordered decreasing chain of order
type $\kappa$ of prime $z$-filters beginning with any non-minimal
prime $z$-filters on $\reals$. (This was shown for
$\kappa\le\omega_1^2$, where $\omega_1$ is the first uncountable
ordinal, in Theorems 8.5, 13.2 and Remark 13.2 of
\cite{mandelker1968}.) We also show that there are various
countable compact subspaces of $\reals$ on which there is a
well-ordered decreasing chain of order type $\continuum$ of prime
$z$-filters ($\continuum$ is identified with the smallest ordinal
of cardinality $\continuum$).

In \cite{mandelker1968}, it was asked whether there exists an
uncountable well-ordered increasing chain of prime $z$-filters on
$\reals$. We shall construct in \S\ref{increasingsection}, for
every ordinal $\kappa$ of cardinality at most $\continuum$, a
well-ordered increasing chain of order type $\kappa$ of prime
$z$-filters on any uncountable locally compact Polish space. We
also construct well-ordered increasing chains of order type
$\continuum$ of prime $z$-filters on various countable compact
subspaces of $\reals$.

All three constructions have as a common ingredient the result due
to Sierpinski that $\naturals$ can be expressed as the union of
$\continuum$ ''almost disjoint'' infinite subsets.

\section{Preliminary definitions and notations}
\label{basicdefinitions}

For details of the theory of the algebras of continuous functions,
see \cite{gillmanjerison1960}.

Let $\Omega$ be a locally compact space; our convention is that
the topological spaces are always Hausdorff. The \emph{one-point
compactification} of $\Omega$ is denoted by
$\onepointcompactification{\Omega}$.

For each prime ideal $P$ in $\C_0(\Omega)$, either there exists a
unique point $p\in\Omega$ such that $f(p)=0$ ($f\in P$), in which
case, we say that $P$ \emph{is supported at the point} $p$; or
otherwise, we say that $P$ \emph{is supported at the (point at)
infinity}.

It is an important fact that, for each prime ideal $P$ in
$\C_0(\Omega)$, the set of prime ideals in $\C_0(\Omega)$ which
contain $P$ is a chain with respect to the inclusion relation.

For each function $f$ continuous on $\Omega$, the \emph{zero set}
of $f$ is denoted by $\zero(f)$. Define
$\zero[\Omega]=\set{\zero(f): \ f\in\C(\Omega)}$. For each closed
subset $Z\subset\Omega$, we have $Z=\zero(f)$ for some function
$f\in\C_0(\Omega)$ if and only if $\Omega\setminus Z$ is
$\sigma$-compact. An ideal $I$ of $\C_0(\Omega)$ is a
\emph{$z$-ideal} if $g\in I$ whenever $g\in\C_0(\Omega)$,
$\zero(g)\supset\zero(f)$ and $f\in I$.

A \emph{$z$-filter} $\F$ on $\Omega$ is a non-empty proper subset
of $\zero[\Omega]$ that is closed under finite intersection and
supersets. Each $z$-filter $\F$ associates with the ideal
\[
    \zero^{-1}[\F]=\set{f\in\C(\Omega):\ \zero(f)\in\F}\quad \textrm{of}\quad \C(\Omega).
\]
A $z$-filter $\Pp$ is a \emph{prime $z$-filter} if $Z_1\cup
Z_2\notin \Pp$ whenever $Z_1,Z_2\in\zero[\Omega]\setminus \Pp$.

Let $\Pp$ be a prime $z$-filter on $\Omega$. Then we say that
$\Pp$ \emph{is supported at a point} $p\in\Omega$, if $p\in Z$ for
each $Z\in\Pp$; if there exists no such $p$, we say that $\Pp$
\emph{is supported at the (point at) infinity}. The support point
of each prime $z$-filter $\Pp$ coincides with the support point of
the prime $z$-ideal $\C_0(\Omega)\cap\zero^{-1}[\Pp]$.

Let $\Omega$ be a compact space. Define
$\level{\Omega}{1}=\limitpoints \Omega$ to be \emph{the set of all
limit points} of $\Omega$. Since $\Omega$ is compact,
$\limitpoints \Omega$ is non-empty unless $\Omega$ is finite. We
then define inductively a non-increasing sequence
$\left(\level{\Omega}{n}:n\in\naturals\right)$ of compact subsets
of $\Omega$ by setting $\level{\Omega}{n+1}
    =\limitpoints\left(\level{\Omega}{n}\right)$ for each $n\in\naturals$.
Set $\level{\Omega}{\infty}
=\bigcap_{n=1}^\infty\level{\Omega}{n}$. By the compactness,
either $\level{\Omega}{\infty}$ is non-empty or
$\level{\Omega}{l}$ is empty for some $l\in\naturals$.

A \emph{Polish space} is a separable completely metrizable space.
Every separable metrizable locally compact space is a Polish
space.

\section{Pseudo-finite families of prime ideals and prime $z$-filters}
\label{psedofinitesection}

Let $\Omega$ be a locally compact space. First, we shall make a
connection between pseudo-finite families of prime ideals and
pseudo-finite families of prime $z$-ideals in $\C_0(\Omega)$. For
each closed subset $E$ of $\Omega$, we define the $z$-ideal
\[
    K_E=\set{f\in\C_0(\Omega):\ E\subset\zero(f)}.
\]
The following is \cite[3.3 and 3.4]{mason1980}, we shall give here
a combined proof.

\begin{lemma}\cite{mason1980}
    Let $I$ be an ideal in $\C_0(\Omega)$. Then
    \[
        I^z=\bigcup\set{K_E:\ E\ \textrm{is closed in}\ \Omega\ \textrm{and}\ K_E\subset I}
    \]
    is the largest $z$-ideal contained in $I$. If $I$ is a prime
    ideal then so is $I^z$.
\end{lemma}
\begin{proof}
    It is easy to see that $I^z$ contains every $z$-ideal contained in $I$. Since the sum of two
    $z$-ideals is again a $z$-ideal,
    we see that
    \[
        I^z=\sum\set{K_E:\ E\ \textrm{is closed in}\ \Omega\ \textrm{and}\ K_E\subset
        I};
    \]
    where the sum is algebraic. Thus $I^z$ is the largest $z$-ideal contained in $I$.

    Now, suppose that $I$ is prime. Let $f_1,f_2\in\C_0(\Omega)\setminus
    I^z$. Then there exists $g_1,g_2\in\C_0(\Omega)\setminus
    I$ such that $\zero(g_i)\supset \zero(f_i)$. Then $\zero(g_1g_2)\supset\zero(f_1f_2)$.
    The primeness of $I$ implies that $g_1g_2\notin I$. So
    $f_1f_2\notin I^z$.
\end{proof}

The following strengthens the implication (a)$\Rightarrow$(c) of
\cite[Lemma 8.4]{pham2005b}.

\begin{proposition}
    Let $\Omega$ be a locally compact space.
    Let $(P_i:\ i\in S)$ be an infinite non-redundant pseudo-finite
    family of prime ideals in $\C_0(\Omega)$. Then $P=\bigcup_{i\in
    S} P_i$ is a prime $z$-ideal, and $(P_i^z:\ i\in S)$ is a non-redundant
    pseudo-finite family of prime $z$-ideals whose union is $P$ such
    that $P_i^z\subset P_i$ ($i\in S$).
\end{proposition}
\begin{proof}
    We shall need another theorem of \cite{mason1980} which say that the sum
    of two non-comparable prime ideals in $\C_0(\Omega)$ is indeed a
    prime $z$-ideal (\cite[3.2]{mason1980}).

    We know that $P$ must be a prime ideal. Assume toward a contradiction that
    $P$ is not a prime $z$-ideal. Choose $\alpha_1\neq\alpha_2\in S$
    arbitrary. Then $P_{\alpha_1}+P_{\alpha_2}$ is a prime $z$-ideal.
    Suppose that we already have distinct indices $\alpha_1,\ldots,
    \alpha_n\in S$ such that $\sum_{i=1}^n
    P_{\alpha_i}$ is a prime $z$-ideal. Then $P\neq\sum_{i=1}^n
    P_{\alpha_i}$, and so we can find $\alpha_{n+1}\in S$ such that $P_{\alpha_{n+1}}\nsubset \sum_{i=1}^n
    P_{\alpha_i}$. The induction can be continued. However, this
    gives a contradiction since then
    \[
        P=\bigcup_{n=1}^\infty
        P_{\alpha_n}=\bigcup_{n=1}^\infty\sum_{i=1}^n P_{\alpha_i}
    \]
    is a $z$-ideal. Hence, $P$ is a prime $z$-ideal.

    We \emph{claim} that $(P_i^z:\ i\in S)$ is a pseudo-finite
    family with union $P$. Indeed, assume toward a contradiction
    that there exists $f\in P$ and distinct $\alpha_n\in S$ ($n\in\naturals$)
    such that $f\notin P_{\alpha_n}^z$ ($n\in\naturals$). For each
    $n$, we can then find $f_n\notin P_{\alpha_n}$ such that
    $\zero(f_n)\supset \zero(f)$; we can further assume that $0\le f_n\le
    2^{-n}$. Define $f_*=\sum_{n=1}^\infty f_n$. Then we see that
    $f_n\le f_*$ so $f_*\notin P_{\alpha_n}$ ($n\in\naturals$),
    and that $\zero(f_*)\supset \zero(f)$ so $f_*\in P$. This is a contradiction to the pseudo-finiteness of
    $(P_i:\ i\in S)$.

    It remains to prove the non-redundancy of $(P_i^z:\ i\in S)$.
    So, assume that $P_\alpha^z\subset P_\beta^z$ for some $\alpha\neq
    \beta\in S$. Then $P_\alpha^z$ is contained in both $P_\alpha$ and
    $P_\beta$, and so $P_\alpha$ and $P_\beta$ are in a chain. This
    contradicts the non-redundancy of $(P_i:\ i\in S)$.
\end{proof}

Conversely, it is obvious that if $(Q_i:\ i\in S)$ is a
pseudo-finite family of prime ($z$-)ideals and $P_i$ is a prime
ideal containing $Q_i$ and contained in $\bigcup_{\alpha\in S}
Q_\alpha$ ($i\in S$) then $(P_i:\ i\in S)$ is a pseudo-finite family
of prime ideals.

We define a similar notion of pseudo-finite families of prime
$z$-filters.

\begin{definition}
An indexed family $(\Pp_i)_{i\in S}$ of prime $z$-filters $\Omega$
is \emph{pseudo-finite} if $Z\in \Pp_i$ for all but finitely many
$i\in S$ whenever $Z\in \bigcup_{i\in S} \Pp_i$.
\end{definition}
A pseudo-finite family of prime $z$-filters $(\Pp_i:\ i\in S)$ is
said to be \emph{non-redundant} if for every proper subset $T$ of
$S$, $\bigcap_{i\in T} \Pp_i\neq \bigcap_{i\in S} \Pp_i$. Similar
to \cite[Lemma 3.4]{pham2005b} we have the following.

\begin{lemma}
    \label{nonredundancy}
    Let $(\Pp_\alpha:\alpha\in S)$ be a pseudo-finite family of prime $z$-filters on $\Omega$.
    Then the following are equivalent:
    \begin{itemize}
        \item[(a)] $(\Pp_\alpha)$ is non-redundant;
        \item[(b)] $\Pp_\alpha \nsubset \Pp_{\beta}\quad(\alpha\neq \beta\in S)$;
        \item[(c)] $\bigcap_{\beta\neq\alpha} \Pp_\beta \nsubset \Pp_{\alpha}$ for each $\alpha\in S$.
    \end{itemize}
\end{lemma}

\begin{proof}
Obviously, (c)$\Rightarrow$(a)$\Rightarrow$(b).

We now prove (b)$\Rightarrow$(c). Fix $\alpha\in S$. By condition
(b), $\Pp_\beta\nsubset \Pp_\alpha$ ($\beta\in
S\setminus\set{\alpha}$). Choose $Z_0\in \Pp_{\beta_0}\setminus
\Pp_\alpha$ for some $\beta_0\in S\setminus\set{\alpha}$. Then, by
the pseudo-finiteness, we have $Z_0\in \Pp_\beta$ for all but
finitely many $\beta\in S$. Let $\beta_1,\ldots, \beta_n$ be those
indices $\beta\in S\setminus\set{\alpha}$ such that $Z_0\notin
\Pp_\beta$. For each $1\le k \le n$, choose $Z_k\in
\Pp_{\beta_k}\setminus \Pp_\alpha$, and set $Z=\bigcup_{k=0}^n
Z_k$. Then $Z\in \Pp_\beta$ ($\beta\in S\setminus\set{\alpha}$),
but $Z\notin \Pp_\alpha$, by the primeness of $\Pp_\alpha$. Thus
(c) holds.
\end{proof}

The following definition and proposition are adapted from
\cite{pham2005b}.

\begin{definition}
\label{extensibility} Let $\Omega$ be a locally compact space, and
let $S$ be a non-empty index set. Let $\F$ be a $z$-filter on
$\Omega$, and let $(Z_\alpha:\ \alpha\in S)$ be a sequence of zero
sets on $\Omega$. Then $\F$ is \emph{extendible} with respect to
$(Z_\alpha:\ \alpha\in S)$ if both the following conditions hold:
    \begin{itemize}
        \item[(a)] $Z_\alpha\notin \F$, and $Z_\alpha\cup Z_\beta\in \F$ ($\alpha\neq \beta\in S$);
        \item[(b)] for each $Z\in \zero[\Omega]$, if $Z\cup Z_{\alpha_{0}}\in \F$
        for some $\alpha_0\in S$, then $Z\cup Z_\alpha\in \F$ for all except
        finitely many $\alpha\in S$.
    \end{itemize}
\end{definition}

\begin{proposition}
    \label{generalconstructionpsf}
    Let $\Omega$ be a locally compact space. Suppose that there exist a $z$-filter $\F$ and
    a family $(Z_\alpha:\ \alpha\in S)$ in $\zero[\Omega]$ such that $\F$ is
    extendible with respect to $(Z_\alpha:\ \alpha\in S)$. Then there exists a
    pseudo-finite family of prime $z$-filters $(\Pp_\alpha:\ \alpha\in S)$ such that
    $Z_\alpha\in\bigcap_{\gamma\neq \alpha} \Pp_\gamma\setminus \Pp_\alpha$ for each $\alpha\in S$.
\end{proposition}
\begin{proof}
    We see that the union of a chain of $z$-filters, each of which contains $\F$
    and is extendible with respect to $(Z_\alpha)$, is also extendible with respect to
    $(Z_\alpha)$. Thus, by Zorn's lemma, we can suppose that $\F$ is a maximal one among
    those $z$-filters.

    For each $\alpha$, set
    $\F_\alpha=\set{Z\in\zero[\Omega]:\ Z\cup Z_\alpha\in\F}$, and set
    $\Pp=\bigcup_{\alpha\in S} \F_\alpha$. By
    the extensibility of $\F$, we see that whenever $Z\in \Pp$ then
    $Z\in \F_\alpha$ for all except finitely many $\alpha\in S$. Thus,
    in particular, the set $\Pp$ is actually a $z$-filter.

    \emph{Claim 1:} For each $Z_0\in \zero[\Omega]\setminus \Pp$, we have
    $\set{Z\in\zero[\Omega]:\ Z\cup Z_0\in \F}=\F$. Indeed, we see that
    $Z_\alpha\notin\G=\set{Z\in\zero[\Omega]:\ Z\cup Z_0\in \F}$ ($\alpha\in S$); for otherwise,
    $Z_0$ would be in $\Pp$.
    It then follows easily that $\G$ is extendible with
    respect to $(Z_\alpha)$. This and the maximality of $\F$ imply the claim.

    \emph{Claim 2:} $\Pp$ is a prime $z$-filter. We have to prove
    that, whenever $Z_1,Z_2\in \zero[\Omega]$ are such that $Z_1\cup Z_2\in \Pp$,
    but $Z_1\notin \Pp$, then $Z_2\in \Pp$. Indeed, let $\alpha_0\in S$ be such that
    $Z_{\alpha_0}\cup Z_1\cup Z_2\in \F$. Then $Z_{\alpha_0}\cup Z_2\in \F$, by
    the first claim, and so $Z_2\in \F_{\alpha_0}$.

    Now, for each $\alpha\in S$, define
    \[
        \D_\alpha=\set{Z_{\alpha}\cup Z:\  Z\in \zero[\Omega]\setminus \Pp}.
    \]
    Then, by Claim 2, the set $\D_\alpha$ is closed under finite union. Obviously,
     $\D_\alpha\cap \F_{\alpha}=\emptyset$. Thus,
    there exists a prime $z$-filter $\Pp_\alpha$ containing $\F_\alpha$ such that
    $\D_\alpha\cap \Pp_\alpha=\emptyset$.

    We see that $\F_\alpha\subset
    \Pp_\alpha\subset \Pp$ and $Z_\alpha\notin \Pp_\alpha$ ($\alpha\in S$).
    The result then follows.
\end{proof}

We now define a ``prototype'' space $\Xi$. Denote by $\infty$ the
point adjoined to $\naturals$ to obtain its one-point
compactification $\onepointcompactification{\naturals}$. The
product space $(\onepointcompactification{\naturals})^{\naturals}$
is a compact metrizable space. Define $\Xi$ to be the compact
subset of $(\onepointcompactification{\naturals})^{\naturals}$
consisting of all elements $(n_1,n_2,\ldots)$ with the property
that there exists $k\in\naturals$ such that $n_i\ge k$ ($1\le i\le
k$) and such that $n_i=\infty$ ($i> k$). The convention is that
$\infty>n$ ($n\in\naturals$).

\begin{lemma}
    \label{embedprototype}
    \cite[Lemma 9.2]{pham2005b} Let $\Omega$ be a locally compact metrizable space. Suppose that there
    exists a point $p\in\level{(\onepointcompactification{\Omega})}{\infty}$. Then there exists
    a homeomorphic embedding
    $\iota$ of $\Xi$ onto a closed subset of $\onepointcompactification{\Omega}$
    such that $\iota(\infty,\infty,\ldots)=
    p$.
\end{lemma}

A key to our construction is the result due to Sierpinski that there
exists a family $\set{E_\alpha:\ \alpha\in\continuum}$ of infinite
subsets of $\naturals$ satisfying the following properties:
\begin{enumerate}
    \item $\naturals=\bigcup_{\alpha\in\continuum} E_\alpha$, and
    \item $E_\alpha\cap E_\beta$ is finite for each $\alpha\neq \beta\in\continuum$.
\end{enumerate}
We sketch the nice construction of such family as follows (cf.
\cite{williams1977}): The set $\naturals$ is isomorphic to
\[
    C=\bigcup_{n=1}^\infty\set{f:\set{1,\ldots,n}\to\set{1,2}}.
\]
For each $f:\naturals\to\set{1,2}$, define
\[
    C_f=\set{\textrm{the restrictions of}\ f\ \textrm{to}\ \set{1,\ldots, n}:\
(n\in\naturals)}.
\]
We see that $C=\bigcup_{f:\naturals\to \set{1,2}} C_f$ and that
$C_f\cap C_g$ is finite for each $f\neq g$. We can then map back
from $C$ to $\naturals$. Inspecting the construction, we see that
$\set{E_\alpha:\ \alpha\in\continuum}$ enjoys the following
property:
\begin{itemize}
    \item[(i')] The cardinality of $\set{\alpha\in\continuum:\ n\in E_\alpha}$ is $\continuum$
        for each $n\in\naturals$.
\end{itemize}

\begin{lemma}
    \label{psfprototype}
    There exists a non-redundant pseudo-finite
    family $\set{\Q_\alpha:\ \alpha\in\continuum}$ of prime $z$-filters
    on $\Xi$ such that each $z$-filter
    is supported at the point $(\infty,\infty,\ldots)$.
\end{lemma}
\begin{proof}
    Let $\set{E_\alpha:\ \alpha\in\continuum}$ be the family of
    infinite subsets of $\naturals$ as in the previous paragraph.
    For each $\alpha\in\continuum$, define
    \[
        N_\alpha=\set{(j_1,j_2,\ldots)\in\Xi:\ j_n=\infty\ (n\in E_\alpha)}.
    \]
    Let $\F$ to be the $z$-filter generated by all $N_\alpha\cup N_\beta$
    ($\alpha,\beta\in\continuum,\, \alpha\ne \beta$). We \emph{claim}
    that $\F$ is extendible with respect to $(N_\alpha:\ \alpha\in\continuum)$; the proof will
    then be completed by applying Lemma \ref{generalconstructionpsf}.

    Obviously, $N_\alpha\cup N_\beta\in \F$ ($\alpha\neq \beta$).
    We \emph{claim} that $N_\alpha\notin \F$ ($\alpha\in \continuum$). Indeed, assume the contrary.
    Then there exist $\gamma_1,\ldots,\gamma_m\in \continuum\setminus\set{\alpha}$
    such that $N_\alpha\supset \bigcap_{i=1}^m N_{\gamma_i}$.
    Since $E_\alpha$ is infinite whereas each $E_\alpha\cap E_{\gamma_i}$ is
    finite, there exists $l\in E_\alpha\setminus\bigcup_{i=1}^m
    E_{\gamma_i}$. We see that $(j_i) \in \bigcap_{i=1}^m
    N_{\gamma_i}\setminus N_\alpha$ where $j_l=l$ and $j_i=\infty$
    ($i\neq l$); a contradiction.

    Finally, suppose that $N\in \zero[\Xi]$ such that
    $N\cup N_{\alpha}\in \F$ for some $\alpha\in \continuum$. Then, there exist
    $\gamma_1,\ldots,\gamma_m\in\continuum\setminus \set{\alpha}$ such that
    \[
        N\cup N_\alpha\supset\bigcap_{i=1}^m N_{\gamma_i},\quad \textrm{and so}\quad
        N\,\supset\,\closure{\bigcap_{i=1}^m N_{\gamma_i}\setminus
        N_\alpha}.
    \]
    As above, there exists $l\in E_\alpha\setminus\bigcup_{i=1}^m
    E_{\gamma_i}$. We can then choose $\gamma_{m+1},\ldots,\gamma_n\in
    \continuum$ such that
    \[
        \set{1,\ldots, l}\subset \bigcup_{i=1}^n E_{\gamma_i}.
    \]
    We \emph{claim} that $N\supset \bigcap_{i=1}^n N_{\gamma_i}$.
    Indeed, let  $(j_i)\in \bigcap_{i=1}^n
    N_{\gamma_i}$. Then  $j_1=\cdots=j_l=\infty$. We see that there exists $k\ge l$
    such that $j_i\ge k$ ($1\le i\le k$) and $j_i=\infty$ ($i> k$).
    For each $r\in\naturals$, set $j^{(r)}_i=j_i$
    ($i\neq l$) and set $j^{(r)}_l=k+r$. Then, we see that
    $(j^{(r)}_i)\in \bigcap_{i=1}^m N_{\gamma_i}\setminus
    N_\alpha\subset N$ and $\lim_r (j^{(r)}_i)=(j_i)$. Thus
    $(j_i)\in N$. Hence, for each $\beta\in\continuum\setminus\set{\gamma_1,\ldots,\gamma_n}$,
    we have $N\cup N_\beta\in \F$.
\end{proof}

\begin{theorem}
    \label{constructpsfprimes}
    Let $\Omega$ be a locally compact metrizable space. Suppose that
    $p\in\level{(\onepointcompactification{\Omega})}{\infty}$. Then there exists
    a non-redundant pseudo-finite family $(\Pp_\alpha:\ \alpha\in \continuum)$ of
    prime $z$-filters on $\Omega$, each $z$-filter is supported at $p$.

    Moreover, by setting
    $P_\alpha=\C_0(\Omega)\cap\zero^{-1}[\Pp_\alpha]$,
    we obtain a non-redundant pseudo-finite family of prime
    $z$-ideals in $\C_0(\Omega)$, each ideal is supported at $p$, such that
    \[\cardinal{\C_0({\Omega})\bigg/\bigcap_{\alpha\in\continuum} P_\alpha} =\continuum.\]
\end{theorem}
\begin{proof}
    In this proof, we shall identify $\Xi$ with a closed subset of
    $\onepointcompactification{\Omega}$ such that $(\infty,\infty,\ldots)$ is identified with
    $p$; in the case where $p\in\Omega$, we can further assume that
    $\Xi\subset\Omega$
    (cf. Lemmas \ref{embedprototype}).

    Let $(\Q_\alpha:\
    \alpha\in\continuum)$ be the family of prime $z$-filters on
    $\Xi$ as constructed in Lemma \ref{psfprototype}. For each
    $\alpha\in\continuum$, set
    \[
        \Pp_\alpha=\set{Z\in\zero[\Omega]:\
        (Z\cup\set{p})\cap\Xi\in\Q_\alpha}.
    \]
    Note that every closed subset of $\Xi$ is in $\zero[\Xi]$, so
    we can see that each $\Pp_\alpha$ is a prime $z$-filter on
    $\Omega$. The pseudo-finiteness of $(\Pp_\alpha:\
    \alpha\in\continuum)$ and of $(P_\alpha:\ \alpha\in \continuum)$
    then follows from that of $(\Q_\alpha)$. The cardinality
    condition follows from the fact that
    $\cardinal{\C(\Xi)}=\continuum$.

    By the non-redundancy of $(\Q_\alpha:\ \alpha\in \continuum)$,
    for each $\alpha\in\continuum$, there exists
    \[
        N_\alpha\in\bigcap_{\beta\neq\alpha} \Q_\beta\setminus \Q_\alpha.
    \]
    By the Urylson's lemma, we can find $Z_\alpha\in\zero[\Omega]$ such that
    $(Z_\alpha\cup\set{p})\cap\Xi=N_\alpha$; we can even require
    $\Omega\setminus Z_\alpha$ to be $\sigma$-compact so that $Z_\alpha=\zero(f_\alpha)$
    for some $f_\alpha\in\C_0(\Omega)$. Thus we see that
    \[
        Z_\alpha\in\bigcap_{\beta\in\continuum,\, \beta\neq\alpha} \Pp_\beta\setminus \Pp_\alpha
        \quad\textrm{and}\quad
        f_\alpha\in\bigcap_{\beta\in\continuum,\, \beta\neq\alpha} P_\beta\setminus P_\alpha.
    \]

    Finally, we shall prove that each $\Pp_\alpha$ (and hence each
    $P_\alpha$) is supported at $p$ ($\alpha\in\continuum$).
    Indeed, in the case where $p$ is the point at infinity of $\Omega$, for each
    $x\in\Omega$, there exists $N\in\Q_\alpha$ such that $x\notin
    N$. We can then find $Z\in\zero[\Omega]$ such that $x\notin
    Z$ and that $(Z\cup\set{p})\cap\Xi=N$. Thus $Z\in\Pp_\alpha$
    and $x\notin Z$. So $\Pp_\alpha$ is support at infinity. On
    the other hand, in the case where $p\in\Omega$, let $Z\in\Pp_\alpha$ be arbitrary.
    Then $Z\cap\Xi$ is closed in $\Xi$, and so it is in
    $\zero[\Xi]$. Since $\set{p}\in\zero[\Xi]\setminus\Q_\alpha$,
    we deduce that $Z\cap \Xi\in\Q_\alpha$. Hence,
    $p\in Z$, and thus $\Pp_\alpha$ is supported at $p$.
\end{proof}

\begin{corollary}
    Let $p\in\onepointcompactification{\reals}$.
    There exists a family $(P_\alpha:\ \alpha\in\continuum)$
    of prime $z$-ideals in $\C_0(\reals)$ with the following properties:
    \begin{itemize}
        \item If $f\in P_{\alpha_0}$ for some $\alpha_0\in\continuum$, then
            $f\in P_\alpha$ for all but finitely many $\alpha\in
            \continuum$;
        \item $\bigcap_{\alpha\neq \alpha_0} P_\alpha \nsubset P_{\alpha_0}$ for each $\alpha_0\in
        \continuum$;
        \item each $P_\alpha$ is supported at $p$. \enproof
    \end{itemize}
\end{corollary}

There are many countable compact metrizable spaces $\Omega$ with
$\level{\Omega}{\infty}\neq \emptyset$. We note as a specific
example the following countable compact subset of $[0,1]$:
\[
    \Delta=\set{0}\cup \set{\sum_{i=1}^k 2^{-n_i}:\ k,\,n_1,n_2,\ldots,n_k\in\naturals\ \textrm{and}\ k\le
    n_1<\cdots<n_k}.
\]

\begin{corollary}
    There exists a family $(P_\alpha:\ \alpha\in\continuum)$
    of non-modular prime $z$-ideals in $\C_0(\Delta\setminus\set{0})$ with the following properties:
    \begin{itemize}
        \item If $f\in P_{\alpha_0}$ for some $\alpha_0\in\continuum$, then
            $f\in P_\alpha$ for all but finitely many $\alpha\in
            \continuum$;
        \item $\bigcap_{\alpha\neq \alpha_0} P_\alpha \nsubset P_{\alpha_0}$ for each $\alpha_0\in
        \continuum$. \enproof
    \end{itemize}
\end{corollary}

\section{Well-ordered decreasing chains of prime ideals and prime $z$-filters}
\label{decreasingsection}

Let $\Omega$ be a metrizable locally compact space. If
$\level{\onepointcompactification{\Omega}}{n}=\emptyset$ for some
$n\in\naturals$, then it can be seen that every chain of prime
$z$-ideals in $\C_0(\Omega)$ or prime $z$-filters on $\Omega$ has
length at most $n$. Hence, in this section we shall suppose that
$\level{\onepointcompactification{\Omega}}{\infty}\neq\emptyset$.

In the following, $\Xi$ is the compact subset of
$(\onepointcompactification{\naturals})^\naturals$ defined in the
previous section. Also, our convention is that $\max \emptyset$ is
smaller and $\min \emptyset$ is bigger than everything, and that
$\bigcap_{\alpha\in\emptyset} N_\alpha$ is the whole space (i.e.
$\Xi$ in the next lemma) and $\bigcup_{\alpha\in\emptyset}
N_\alpha=\emptyset$.

\begin{lemma}
    \label{almost_disjoint_continuum}
    There exists a family $(N_\alpha)_{\alpha\in\continuum}$ of zero
    sets on
    $\Xi$ satisfying that, for every $\gamma\in\continuum$ and disjoint finite subsets
        $F$ and $G$ of $\continuum$, we can find a finite subset $H$ of
        $\continuum$
        with the properties that $\gamma\le\min H$ and that
        \[
            \closure{\bigcap_{\beta\in G}N_\beta\setminus\Big(\bigcup_{\alpha\in F} N_\alpha\Big)}
            \supset \bigcap_{\beta\in G\cup H}N_\beta.
        \]
\end{lemma}

\begin{proof}
Recall from the previous section that there exists a family
$(E_\alpha:\ \alpha\in\kappa)$ of infinite subsets of $\naturals$
satisfying:
\begin{itemize}
    \item[(a)] $E_\alpha\cap E_\beta$ is finite for each $\alpha\neq \beta\in\kappa$, and
    \item[(b)] the cardinality of $\set{\alpha\in\continuum:\ n\in E_\alpha}$ is $\continuum$
    ($n\in\naturals$).
\end{itemize}
Similar to Lemma \ref{psfprototype}, we define, for each
$\alpha\in\continuum$,
\[
        N_\alpha=\set{(j_1,j_2,\ldots)\in\Xi:\ j_n=\infty\ (n\in E_\alpha)}.
\]

Let $G$ be a finite subset of $\kappa$ and let
$\alpha,\gamma\in\continuum$. Then, there exists $l\in
E_\alpha\setminus\bigcup_{\beta\in G} E_{\beta}$. Since the
cardinality of $\set{\beta:\ \beta<\gamma}$ is less than
$\continuum$, by (b) above, we can find a finite subset $H$ of
$\continuum$ such that $\gamma\le\min H$ and that
\[
    \set{1,\ldots, l}\subset\bigcup_{\beta\in G\cup H} E_\beta.
\]
Then, similar to Lemma \ref{psfprototype}, we see that
\[
        \closure{\bigcap_{\beta\in G}N_\beta\setminus N_\alpha}
        \supset\bigcap_{\beta\in G\cup H}N_\beta.
\]
The general case follows by induction.
\end{proof}

\begin{theorem}
    \label{decreasing_chain_continuum}
    Let $\Omega$ be a metrizable locally compact space, and
    let $p\in\level{\onepointcompactification{\Omega}}{\infty}$.
    Then there exists a well-ordered decreasing
    chain $(\Q_\alpha:\ \alpha\in \continuum)$ of prime $z$-filters on
    $\Omega$ each supported at $p$.

    Furthermore, by setting $Q_\alpha=\C_0(\Omega)\cap\zero^{-1}[\Q_\alpha]$,
    we obtain a well-ordered decreasing chain $(Q_\alpha:\ \alpha\in \continuum)$ of prime
    $z$-ideals in $\C_0(\Omega)$ each supported at $p$.
\end{theorem}
\begin{proof}
Similar to Theorem \ref{constructpsfprimes}, we shall identify
$\Xi$ with a closed subset of $\onepointcompactification{\Omega}$
such that $(\infty,\infty,\ldots)$ is identified with $p$; in the
case where $p\in\Omega$, we can further assume that
$\Xi\subset\Omega$ (cf. Lemmas \ref{embedprototype}).

Let $(N_\alpha:\ \alpha\in\continuum)$ be the family of zero sets
on $\Xi$ as constructed in Lemma \ref{almost_disjoint_continuum}.
For each $\alpha\in\continuum$, choose $Z_\alpha=\zero(f_\alpha)$
for some $f_\alpha\in\C_0(\Omega)$ such that
$(Z_\alpha\cup\set{p})\cap \Xi=N_\alpha$. Also, define
\[
    \F_\alpha=\set{Z\in\zero[\Omega]:\ Z\cup\set{p}\supset \bigcap_{i=1}^n
    N_{\beta_i}\ \textrm{for some}\ \alpha\le\beta_1,\ldots,\beta_n\in\continuum}.
\]
Then $(\F_\alpha)$ is a decreasing $\continuum$-sequence of
$z$-filters on $\Omega$; $Z_\alpha\in\F_\alpha$ but
$Z_\alpha\notin\F_\beta$ ($\alpha<\beta\in\continuum$).

Set
\[
    \D_*=\set{Z\in\zero[\Omega]: \ (Z\cup\set{p})\cap \Xi=\set{p}}.
\]
Then $\D_*$ is closed under taking finite union. Also, since
$\D_*\cap \F_0=\emptyset$, there exists a prime $z$-filter $\Q_0$
on $\Omega$ containing $\F_0$ such that $\Q_0\cap \D_*=\emptyset$.
Let $\gamma\in\continuum$. Suppose that we have already
constructed a well-ordered decreasing chain $(\Q_\alpha:\
\alpha<\gamma)$ of prime $z$-filters on $\Omega$ such that
$\F_\alpha\subset \Q_\alpha$ ($\alpha<\gamma$). If $\gamma$ is a
limit ordinal, set $\Q_\gamma=\bigcap_{\alpha<\gamma}\Q_\alpha$.
Consider now the case where $\gamma=\alpha+1$ for some $\alpha$.
Set
\[
    \D_\gamma=\set{Z\cup Z_\alpha:\ Z\in\zero[\Omega]\setminus\Q_\alpha}.
\]
Then $\D_\gamma$ is closed under taking finite union. Also, we
have $\F_\gamma\cap \D_\gamma=\emptyset$; since otherwise, there
exist $Z\in\zero[\Omega]\setminus\Q_\alpha$ and a finite subset
$G$ of $\continuum$ such that $\gamma\le \min G$ and that
\[
    Z\cup Z_\alpha\cup\set{p}\supset\bigcap_{\beta\in G}
    N_{\beta}\quad\textrm{which implies that}\quad
    Z\cup\set{p}\supset\closure{\bigcap_{\beta\in G}
    N_{\beta}\setminus N_\alpha}\supset\bigcap_{\beta\in H}
    N_{\beta}
\]
for some finite subset $H$ of $\continuum$ with $\gamma\le \min
H$, by Lemma \ref{almost_disjoint_continuum}, or
$Z\in\F_\gamma\subset \Q_\alpha$ a contradiction. Therefore, there
exists a prime $z$-filter $\Q_\gamma$ such that $\F_\gamma\subset
\Q_\gamma$ and $\Q_\gamma\cap\D_\gamma=\emptyset$. We see that, in
this case, $\Q_\gamma\subsetneq \Q_\alpha$ and $Z_\alpha\notin
\Q_\gamma$. Thus, in both cases, the construction can be continued
inductively.

Setting $Q_\alpha=\C_0(\Omega)\cap\zero^{-1}[\Q_\alpha]$. Then
$f_*\notin Q_0$ for $f_*\in\C_0(\Omega)$ such that
$\zero(f)\in\D_*$, and, for each $\gamma=\alpha+1\in\continuum$,
we have $f_\alpha\in Q_\alpha\setminus Q_\gamma$. It follows that
the chain $(Q_\alpha:\ \alpha\in \continuum)$ is decreasing.

The statement on support point follows from the fact that
$\bigcap_{\alpha\in\continuum} N_\alpha=\set{p}$.
\end{proof}

It was proved in \cite[Theorem 13.2]{mandelker1968} that starting
from any non-minimal prime $z$-filter containing a countable zero
set on $\reals$ there exists a well-ordered decreasing full
$\omega_1$-sequence of prime $z$-filters such that each prime
$z$-filter contains a countable zero set. However, besides that
$\omega_1<\continuum$ in the absence of the Continuum Hypothesis,
the union of those countable zero sets are not countable, and thus
that $\omega_1$-sequence says nothing about uncountable chains of
prime $z$-filters on countable spaces.

\begin{corollary}
    Let $\Delta$ be any countable compact subset of $\reals$ such that
    $\level{\Delta}{\infty}\neq\emptyset$. There exists a well-ordered decreasing chain of order type $\continuum$ of prime
    $z$-filters on $\reals$ such that each prime $z$-filter
    contains $\Delta$. \enproof
\end{corollary}

We now look for longer chains. We shall need to restrict to
uncountable locally compact Polish spaces. Note that for any
well-ordered decreasing chain of order type $\kappa$ of prime
$z$-filters on $\Omega$ or prime ideals in $\C_0(\Omega)$, where
$\Omega$ is in addition $\sigma$-compact, $\kappa$ must have
cardinality at most $\continuum$.

\begin{lemma}
    \label{almost_disjoint}
    Let $\kappa$ be an ordinal of cardinality $\continuum$.
    There exists a family $(N_\alpha)_{\alpha\in\kappa}$ of zero
    sets on
    $(\onepointcompactification{\naturals})^\naturals$ such that
    for every disjoint finite subsets $F$ and $G$ of $\kappa$, we have
    \[
        \closure{\bigcap_{\alpha\in F}N_\alpha\setminus\Big(\bigcup_{\beta\in G} N_\beta\Big)}
        =\bigcap_{\alpha\in F}N_\alpha.
    \]
\end{lemma}
\begin{proof}
    Similar to (but simpler than) that of Lemma \ref{almost_disjoint_continuum}.
\end{proof}

\begin{theorem}
    Let $\Omega$ be an uncountable locally compact Polish space. Let $\kappa$ be
    an ordinal of cardinality $\continuum$. Then there exists a well-ordered decreasing
    chain $(\Q_\alpha:\ \alpha\in \kappa)$ of prime $z$-filters on
    $\Omega$.

    Furthermore, by setting $Q_\alpha=\C_0(\Omega)\cap\zero^{-1}[\Q_\alpha]$,
    we obtain a well-ordered decreasing chain $(Q_\alpha:\ \alpha\in \kappa)$ of prime
    $z$-ideals in $\C_0(\Omega)$.
\end{theorem}
\begin{proof}
Every uncountable Polish space contains a closed subsets
homeomorphic to the Cantor space $\set{0,1}^\naturals$, which in
turn contains a copy of
$(\onepointcompactification{\naturals})^\naturals$. Thus,  we
shall identify $(\onepointcompactification{\naturals})^\naturals$
with a closed subset of $X$ of $\Omega$ where
$(\infty,\infty,\ldots)$ is identified with some point $p$. Let
$(N_\alpha:\ \alpha\in\kappa)$ be the family of zero sets on $X$
as constructed in Lemma \ref{almost_disjoint}. Define
\[
    \F_\alpha=\set{Z\in\zero[\Omega]:\ Z\supset \bigcap_{i=1}^n
    N_{\beta_i}\ \textrm{for some}\ \alpha\le\beta_1,\ldots,\beta_n\in\kappa}.
\]
The remaining of the proof is similar to that of Theorem
\ref{decreasing_chain_continuum}, but applying Lemma
\ref{almost_disjoint} instead of Lemma
\ref{almost_disjoint_continuum}.
\end{proof}

\begin{corollary}
    Let $\kappa$ be any ordinal of cardinality $\continuum$. Then:
    \begin{enumerate}
        \item There exists a well-ordered decreasing chain of order type $\kappa$
         of prime $z$-filters on $\reals$ starting from any non-minimal prime $z$-filter.
        \item There exists a well-ordered decreasing chain of order type $\kappa$ of prime
    $z$-ideals in $\C_0(\reals)$ starting from any non-minimal prime $z$-ideals.
    \end{enumerate}
\end{corollary}
\begin{proof}
    (i) follows from the theorem and \cite[Theorem
    12.8]{mandelker1968}, and (ii) follows from (i) and the fact that every zero set on $\reals$
    is the zero set of a function in $\C_0(\reals)$.
\end{proof}

\section{Well-ordered increasing chains of prime ideals and prime $z$-filters}
\label{increasingsection}

Let $\Omega$ be a metrizable locally compact space. Similar to the
previous section we shall only consider the case where
$\level{\onepointcompactification{\Omega}}{\infty}\neq\emptyset$.
Recall that there are many countable compact space satisfying this
condition. First we shall prove a general construction.

\begin{definition}
    Let $\kappa$ be any ordinal, and let $(Z_\alpha:\
    \alpha\in\kappa)$ be a family of zero sets on $\Omega$. A zero
    set $Z$ is said to \emph{have property (A)} (\emph{with respect to the
    family $(Z_\alpha: \ \alpha\in\kappa)$}) if for every (possibly empty) finite
    subset $F$ of $\kappa$  and every $\beta\in\kappa$ with $\max F<\beta$
    then
    \[
        Z\cap\bigcap_{\alpha\in F} Z_\alpha\nsubset
        Z_\beta.
    \]

    A zero set $Z$ is said to \emph{have property (B)} (\emph{with respect to the
    family $(Z_\alpha: \ \alpha\in\kappa)$}) if whenever $Z=\bigcup_{i=1}^n Z_i$
    for some $Z_1,\ldots, Z_n\in\zero[\Omega]$ then
    there exists $1\le k\le n$ such that $Z_k$ has property (A).
\end{definition}

\begin{lemma}
    \label{increasing_construction}
    Let $\kappa$ be any ordinal, and let $(Z_\alpha:\
    \alpha\in\kappa)$ be a family of zero sets on $\Omega$.
    Suppose that $\F$ is a $z$-filter on $\Omega$ such that every
    element of $\F$ has property (B) with respect to $(Z_\alpha: \ \alpha\in\kappa)$.
    Then there exists a well-ordered increasing chain $(\Q_\alpha:\
    \alpha\in \kappa)$ of prime $z$-filters containing $\F$ such that $Z_\alpha\notin \Q_\alpha$
    but $Z_\alpha\in \Q_\beta$ ($\alpha<\beta\in\kappa$).
\end{lemma}
\begin{proof}
    Let $\D$ be the collection of all zero sets not having
    property (B). Then obviously $\D$ is closed under
    finite union and $(Z_\alpha:\ \alpha\in\kappa)\subset \D$.
    Since $\F\cap\D=\emptyset$, there exists a prime $z$-filter
    $\Q_0$ such that $\F\subset \Q_0$ and $\Q_0\cap \D=\emptyset$.
    We then define $\Q_\alpha$ to be the $z$-filter generated by
    $\Q_0$ and $\set{Z_\gamma: \gamma<\alpha}$. It follows that
    $Q_\alpha$ is a prime $z$-filter, $\Q_\alpha\subset\Q_\beta$
    and $Z_\alpha\in \Q_\beta$ ($\alpha<\beta\in\kappa$). We need
    to show that $Z_\alpha\notin \Q_\alpha$ ($\alpha\in\kappa$)
    (and thus $(\Q_\alpha)$ is increasing). Assume towards
    a contradiction that $Z_\alpha\in \Q_\alpha$
    for some $\alpha\in\kappa$. Then, there exist $N\in\Q_0$ and
    a finite subset $F$ of $\set{\gamma: \gamma< \alpha}$ such
    that
    \[
        Z_\alpha\supset N\cap\bigcap_{\gamma\in F} Z_\gamma.
    \]
    This implies that $N\in \D$ a contradiction.
\end{proof}

\begin{lemma}
    \label{almost_disjoint_continuum2}
    There exists a family $(N_\alpha)_{\alpha\in\continuum}$ of zero
    sets on
    $\Xi$ satisfying that $\Xi$ has property (B) with respect to
    $(N_\alpha:\ \alpha\in \continuum)$.
\end{lemma}
\begin{proof}
Let $(E_\alpha:\ \alpha\in\continuum)$ and $(N_\alpha:\
\alpha\in\continuum)$ be defined as in Lemma
\ref{almost_disjoint_continuum}.

We shall prove a little stronger statement. Assume towards a
contradiction that there exists $Z_1,\ldots, Z_n\in\zero[\Xi]$
such that $\Xi=\bigcup_{i=1}^nZ_i$ and that, for each $1\le i\le
n$, there exist finite subsets $F_i$ and $G_i$ of $\kappa$ with
$\max F_i<\min G_i$ such that
\[
    Z_i\cap\bigcap_{\alpha\in F_i} N_\alpha\subset \bigcup_{\beta\in G_i}
    N_{\beta}.
\]
Without loss of generality, we can suppose that
\[
    \max F_1\le\max F_2\le\ldots \max F_n.
\]
Fix $\gamma\in\continuum$ such that $\gamma>\max G_i$ ($1\le i\le
n$). We shall prove by induction that there exist finite subsets
$H_i$ of $\continuum$ with $\gamma\le \min H_i$ such that
\[
    \bigcap_{i=1}^{k-1}\bigcap_{\alpha\in F_i\cup H_i} N_\alpha\subset
    \bigcup_{i=k}^n Z_i \qquad (1\le k\le n+1).
\]
This is obviously true when $k=1$ since both sides are $\Xi$.
Suppose that the above is true for some $k<n$. Then we see that
\[
    \bigcap_{i=1}^{k-1}\bigcap_{\alpha\in F_i\cup H_i} N_\alpha\cap
    \bigcap_{\alpha\in F_k} N_\alpha\setminus \bigcup_{\beta\in G_k}
    N_{\beta}\subset \bigcup_{i=k+1}^n Z_i.
\]
Because
\[
    G_k\cap\bigg(\bigcup_{i=1}^k F_i\cup\bigcup_{i=1}^{k-1}
    H_i\bigg)=\emptyset,
\]
by Lemma \ref{almost_disjoint_continuum}, there exists a finite
subset $H_k$ of $\continuum$ such that $\min H_k>\gamma$ and that
\[
    \bigcap_{i=1}^{k}\bigcap_{\alpha\in F_i\cup H_i} N_\alpha\subset
    \closure{\bigcap_{i=1}^{k-1}\bigcap_{\alpha\in F_i\cup H_i} N_\alpha\cap
    \bigcap_{\alpha\in F_k} N_\alpha\setminus \bigcup_{\beta\in G_k}
    N_{\beta}}\subset \bigcup_{i=k+1}^n Z_i.
\]
Thus, the induction can be continued, and so, for $k=n+1$, we have
\[
    \bigcap_{i=1}^{n}\bigcap_{\alpha\in F_i\cup H_i} N_\alpha\subset
    \emptyset;
\]
this is a contradiction.
\end{proof}

\begin{theorem}
    \label{increasing_chain_continuum}
    Let $\Omega$ be a metrizable locally compact space, and
    let $p\in\level{\onepointcompactification{\Omega}}{\infty}$.
    Then there exists a well-ordered increasing chain
    $(\Q_\alpha:\ \alpha\in \continuum)$ of prime $z$-filters on
    $\Omega$ each supported at $p$.

    Furthermore, by setting $Q_\alpha=\C_0(\Omega)\cap\zero^{-1}[\Q_\alpha]$,
    we obtain a well-ordered increasing chain $(Q_\alpha:\ \alpha\in \continuum)$ of prime
    $z$-ideals in $\C_0(\Omega)$ each supported at $p$.
\end{theorem}
\begin{proof}
Similar to Theorem \ref{constructpsfprimes}, we shall identify
$\Xi$ with a closed subset of $\onepointcompactification{\Omega}$
such that $(\infty,\infty,\ldots)$ is identified with $p$; in the
case where $p\in\Omega$, we can further assume that
$\Xi\subset\Omega$ (cf. Lemmas \ref{embedprototype}).

Let $(N_\alpha:\ \alpha\in\continuum)$ be the family of zero sets
on $\Xi$ as constructed in Lemma \ref{almost_disjoint_continuum2}.
For each $\alpha\in\continuum$, choose $Z_\alpha=\zero(f_\alpha)$
for some $f_\alpha\in\C_0(\Omega)$ such that
$(Z_\alpha\cup\set{p})\cap \Xi=N_\alpha$. It follows that $\Omega$
has property (B) with respect to $(Z_\alpha:\
\alpha\in\continuum)$. Thus, by Lemma
\ref{increasing_construction} where $\F=\set{\Omega}$, there
exists a well-ordered  increasing chain $(\Q_\alpha:\
    \alpha\in \continuum)$ of prime $z$-filters such that $Z_\alpha\notin \Q_\alpha$
    but $Z_\alpha\in \Q_\beta$ ($\alpha<\beta\in\continuum$).

The rest is similar to Theorem \ref{decreasing_chain_continuum}.
\end{proof}

\begin{corollary}
    Let $\Delta$ be any countable compact subset of $\reals$ such that
    $\level{\Delta}{\infty}\neq\emptyset$. There exists a well-ordered
    increasing chain of order type $\continuum$ of prime
    $z$-filters on $\reals$ such that each prime $z$-filter
    contains $\Delta$. \enproof
\end{corollary}

For longer chains, as in \S\ref{decreasingsection}, we need to
restrict to uncountable locally compact Polish spaces. Again, in
the case where $\Omega$ is in addition $\sigma$-compact, it will
restrict the ordinal $\kappa$ under consideration to have
cardinality at most $\continuum$.

\begin{lemma}
    \label{almost_disjoint2}
    Let $\kappa$ be an ordinal of cardinality $\continuum$.
    There exists a family $(N_\alpha)_{\alpha\in\kappa}$ of zero
    sets on $(\onepointcompactification{\naturals})^\naturals$
     such that $(\onepointcompactification{\naturals})^\naturals$ has property (B) with respect to
    $(N_\alpha:\ \alpha\in \kappa)$.
\end{lemma}
\begin{proof}
    Similar to Lemma \ref{almost_disjoint_continuum2}, here we apply Lemma \ref{almost_disjoint}
    instead of Lemma \ref{almost_disjoint_continuum}.
\end{proof}

\begin{theorem}
    Let $\Omega$ be an uncountable locally compact Polish space. Let $\kappa$ be
    an ordinal of cardinality $\continuum$. Then there exists a well-ordered increasing
    chain $(\Q_\alpha:\ \alpha\in \kappa)$ of prime $z$-filters on
    $\Omega$.

    Furthermore, by setting $Q_\alpha=\C_0(\Omega)\cap\zero^{-1}[\Q_\alpha]$,
    we obtain a well-ordered  increasing chain $(Q_\alpha:\ \alpha\in \kappa)$ of prime
    $z$-ideals in $\C_0(\Omega)$.
\end{theorem}
\begin{proof}
Similar to previous proofs.
\end{proof}

\begin{corollary}
    Let $\kappa$ be any ordinal of cardinality $\continuum$. Then
    there exists a well-ordered increasing chain of order type $\kappa$ of prime $z$-filters on
    $\reals$. \enproof
\end{corollary}


\providecommand{\bysame}{\leavevmode\hbox
to3em{\hrulefill}\thinspace}
\providecommand{\MR}{\relax\ifhmode\unskip\space\fi MR }
\providecommand{\MRhref}[2]{%
  \href{http://www.ams.org/mathscinet-getitem?mr=#1}{#2}
} \providecommand{\href}[2]{#2}

\end{document}